\documentclass[10pt,ps]{amsart}
\usepackage{amsmath}
\usepackage{amsthm}
\usepackage{amssymb}
\usepackage{amscd}
\usepackage{amsfonts}
\usepackage{amsbsy}

\newtheorem {theorem} {Theorem}
\newtheorem {proposition} [theorem]{Proposition}
\newtheorem {corollary} [theorem]{Corollary}

\newtheorem {conjecture} [theorem]{Conjecture}

\newcommand{\R}{\mathbb{R}}

\newcommand{\C}{\mathbb{C}}

\newcommand{\X}{\mathcal X}

\begin{document}
\title[Quadratic polynomial
differential systems with algebraic solutions ] {A family of
quadratic polynomial differential systems with algebraic solutions
of arbitrary high degree}
\date{}
\dedicatory{} \maketitle
\centerline{R. Ram\'{\i}rez$^1$, V. Ram\'{\i}rez$^2$ }

\smallskip

\centerline{$^1$ Departament d'Enginyeria Inform\`{a}tica i
Matem\`{a}tiques.} \centerline{ Universitat Rovira i Virgili.}
\centerline{ Avinguda dels Pa\"{\i}sos Catalans 26, 43007 Tarragona,
Spain.} \centerline{E-mail:rafaelorlando.ramirez@urv.cat}

\smallskip

 \centerline{$^2$ Universitat Central de
Barcelona,} \centerline{ Gran V\'{\i}a de las Cortes Catalanas, 585
08007 Barcelona,
 Spain.}
\centerline{E-mail:vramirsa8@alumnes.ub.edu}



\begin{abstract}
We show that the algebraic curve $a_0(x)(y-r(x))+p_2(x)a'(x)=0,$
where $r(x)$ and $p_2(x)$ are polynomial of degree 1 and 2
respectively and $a_0(x)$ is a polynomial solution of the convenient
Fucsh's equation, is an invariant curve of the  quadratic planar
differential system. We study the particular case when $a_0(x)$ is
an orthogonal polynomials. We prove that that in this case the
quadratic differential system is Liouvillian integrable.
\end{abstract}
\section{Introduction and statement of the main results}

Consider the set $\Sigma$ of all planar real polynomial vector
fields $\X=(P,Q)$ associated to the differential polynomial systems
\begin{equation}\label{KLK}
\dot{x}=P(x,y),\quad \dot{y}=Q(x,y).
\end{equation}
of degree $m=\max\left\{\mbox{deg{P}},\,\mbox{deg{Q}}\right\},$ here
the dot denotes derivative respect to time $t.$

\smallskip

Let $U$ be an open and dense set in $\R^2$. We say that a
non-constant $C^1$ function $H \colon U \to \R$ is a \emph{first
integral} of the polynomial vector field $\X$ on $U$, if
$H(x(t),y(t))$ is constant for all values of $t$ for which the
solution $(x(t),y(t))$ of $\X$ is defined on $U$. Clearly $H$ is a
first integral of $\X$ on $U$ if and only if $\X H=0$ on $U$.

\smallskip

Let $\C[x,y]$ be the ring of all complex polynomials in the
variables $x$ and $y$, and let $\X$ be a polynomial vector field of
degree $m$, and let $g=g(x,y)\in \C[x,y]$. Then $g=0$ is an {\it
invariant algebraic curve} of $\X$ if
\begin{equation}\label{KLK1}
\X g=P\dfrac{\partial g}{\partial x}+Q\dfrac{\partial g}{\partial
y}=K g,
\end{equation}
where $K=K(x,y)$ is a polynomial of degree at most $m-1$, which is
called the {\it cofactor} of $g=0$ (for more details see for
instance \cite{Ll}). . If the polynomial $g$ is irreducible in
$\C[x,y]$, then we say that the invariant algebraic curve $g=0$ is
{\it irreducible} and that its {\it degree} is the degree of the
polynomial $g$. We work with real polynomial vector fields but for
these vector fields we consider complex invariant algebraic curves
because sometimes the existence of a real first integral is forced
by the existence of complex invariant algebraic curves, for more
details on the so--called Darboux theory of integrability see for
instance the Chapter 8 of \cite{DLA}.

\smallskip

A non-constant function $R \colon U \to \R$ is an \emph{integrating
factor} for the polynomial vector field $\X$, if one of the
following three equivalent conditions holds
\begin{equation}\label{eq:triangle}
\frac{\partial (RP)}{\partial x} =-\frac{\partial(RQ)}{\partial y},
\quad \text{div}\, (RP,RQ) =0, \quad \X R =-R \, \text{div} \, (P,Q)
\end{equation}
on $U$. As usual the \emph{divergence} of the vector field $\X$ is
defined by
\[
\text{div} \,(P,Q) = \frac{\partial P}{\partial x } + \frac{\partial
Q}{\partial y}.
\]
Knowing an integrating factor $R$ of the differential system
\eqref{KLK} we can compute a first integral $H$ of $\X$ as follows
\[
H=-\int R(x,y) P(x,y)\, d y + h(x),
\]
where the function $h(x)$ is determined from the equality
$\displaystyle \frac{\partial H}{\partial x} = R(x,y) Q(x,y)$.

\smallskip

Let $f_i,g_j,h_j \in \C[x,y]$ for $i=1,\ldots,p$ and $j=1,\ldots,q$.
Then the (multi--valued) function
\begin{equation}\label{eq:10}
f_1^{\lambda_1} \cdots f_p^{\lambda_p} e^{\mu_1 g_1/h_1} \cdots
e^{\mu_q g_q/h_q}
\end{equation}
with $\lambda_i,\mu_j \in \C$ is called a \emph{(generalized)
Darboux function.}

\smallskip

We say that a polynomial differential system \eqref{KLK} is
\emph{Liouvillian integrable} if it has a first integral or an
integrating factor given by a generalized Darboux function, for more
details see Singer \cite{S}.

\smallskip

 The function $H=H(x,y)$ defined
in an open subset $\tilde{\textsc{D}}_1$ of $\textsc{D}$ such that
its closure coincides with $\textsc{D}$ is called a {\it first
integral} if it is constant on the solutions of system \eqref{00}
contained in $\tilde{\textsc{D}}_1$, i.e. $\X(H)|_{\tilde{
\textsc{D}}_1}=0$.

\smallskip

{F}rom Jouanolou's Theorem (see for instance \cite{Jou,Ll}) it
follows that for a given polynomial differential system of degree
$m$ the maximum \textit{degree of its irreducible invariant
algebraic curves} is bounded, since either it has a finite number
$p<\dfrac{1}{2}m(m+1)+2,$
 of invariant algebraic curves, or all its trajectories are
contained in invariant algebraic curves and the system admits a
\textit{rational first integral}. Thus for each polynomial system
there is a natural number $N$ which bounds the degree of all its
irreducible invariant algebraic curves. A natural question, going
back to Poincar\'e ( for more details see \cite{Po}), is to give an
effective procedure to find $N.$ Partial answer to this question
were given in  \cite{Campilo,Ca,CLN}. Of course, given such a bound,
it is easy to compute the algebraic curves of the system.

\smallskip

Unfortunately, for the class of polynomial systems with fixed degree
$m$, there does not, exist a uniform upper bound for $N$ (see for
instance \cite{Ll}).

\smallskip

Another common suggestion is there is some number $M(m)$ for which
all polynomials systems of degree $m$ with invariant algebraic
curves of degree greatest that $M(m)$ has a rational first integral
(see for instance \cite{CLL}). In \cite{CLL} and \cite{Ol} was
proved that no such function $M(m)$ exist.

\smallskip

The aim of this paper, by developing the idea given in \cite{CLL},
is to give a family of polynomials systems of degree 2 without
rational first integral but with irreducible invariant algebraic
curves of arbitrary high degree.

\smallskip

The following  results we can find in \cite{LR0,LRS}).
\begin{proposition}\label{Nat}
 The most general differential
system whit invariant algebraic curve $g=g(x,y)=0$ is
\[\dot{x}=-\nu\dfrac{\partial g}{\partial y}+\lambda_1g,\quad \dot{y}=\nu\dfrac{\partial g}{\partial
x}+\lambda_2g,\] where $\nu,\,\lambda_1$ and $\lambda_2$ are
arbitrary functions.
\end{proposition}

\begin{theorem}\label{PP1}
The  algebraic curve $g=a_0(x)y+a_1(x)=0$ is invariant for the
quadratic system
\begin{equation}\label{Nat1}
\begin{array}{rl}
\dot{x}=&p_{22}x^2+p_{21}x+p_{20}:=p_2(x),\vspace{0.2cm}\\
\dot{y}=&q_0y^2+(q_{11}x+q_{10})y+q_{22}x^2+q_{21}x+q_{20}:=q_0y^2+q_1(x)y+q_{2}(x),
\end{array}
\end{equation}
 with cofactor $K=\alpha y+\beta x+\gamma$ if
and only if \[ \alpha=q_0,\quad \alpha
a_1(x)=p_2a'_0(x)-\left((\beta
-q_{11})x+\gamma-q_{10}\right)a_0(x),\] and $a_0=a_0(x)$ is a
polynomial solution of the \textit{ Fucsh's equation}

\begin{equation}\label{Rf2}
\begin{array}{rl}
&w''+\dfrac{\left(2p_{22}+q_{11}-2\beta\right)x+2p_{21}+q_{10}-2\gamma}{p_{22}x^2+p_{21}x+p_{20}}\,w'-
\dfrac{q_{11}-\beta}{p_{22}x^2+p_{21}x+p_{20}}\,w\vspace{0.2cm}\\
&-\dfrac{\left(\beta^2+\alpha q_{22}-\beta
q_{11}\right)x^2+\left(\alpha q_{21}+2\gamma
\beta-q_{10}\beta-\gamma q_{11}\right)x+\gamma^2-\gamma
q_{10}+\alpha q_{20}}{\left(p_{22}x^2+p_{21}x+p_{20}\right)^2}\,w=0
\end{array}
\end{equation}
\end{theorem}
 Henceforth we shall consider that
$q_0=\alpha=1.$

\smallskip

Our main results are the following
\begin{theorem}\label{A1}
Under the assumptions of Theorem \ref{PP1} we obtain that the most
general quadratic polynomial differential system which admits the
invariant algebraic curve
\begin{equation}\label{ORT}
a_0\left(y-\left((\beta-q_{11})x+\gamma-q_{10}\right)\right)+p_2a'_0=0,
\end{equation}
 where
$a_0=a_0(x)$ is a solution of the equation
\begin{equation}\label{Rr01}
 p_2(x)a''_0+ r(x)a'_0+{\kappa} a_0=0.
\end{equation}
where $\kappa$ is a nonzero constant and
$r(x)=\left(2p_{22}+q_{11}-\beta\right)x+\gamma+p_{21}-q_{10},$ is
the system
\begin{equation}\label{Rf4}
\begin{array}{rl}
\dot{x}=&p_{22}x^2+p_{21}x+p_{20},\vspace{0.2cm}\\
\dot{y}=&q_0y^2+\left(2\beta-2p_{22}+\tau_{11}\right)xy+\left( 2\gamma+\tau_{10}-p_{21} \right)y\vspace{0.2cm}\\
&+\left(2p^2_{22}-\left(\tau_{11}-\tau_0+3\beta\right)p_{22}+\beta^2+\beta\tau_{11}\right))x^2\vspace{0.2cm}\\
&+\left(2p_{22}-\left(2\beta-\tau_0+\tau_{11}\right)p_{21}+
\gamma\left(\tau_{11}-2p_{22}+2\beta\right)+\beta\tau_{10}\right)x\vspace{0.2cm}\\
&+p_{20}\left(2p_{22}-\tau_{11}+\tau_0-\beta\right)+\gamma\left(\gamma+\tau_{10}-p_{21}\right),
\end{array}
\end{equation}
where $\tau_{11},\,\tau_{21}$ and $\tau_{20}$ are convenient
constants.
\end{theorem}
The problem which appear is to deduce the conditions under which
Fucsh equation \eqref{Rf2} admits polynomial solutions. We shall
study the subcase when the sought polynomials are hypergemetric
functions (with convenient conditions on the parameters) and
orthogonal polynomials.

\smallskip

 A very important subcase of \eqref{Rr01} is
the hypergeometric differential equation
\begin{equation}\label{HypEq}
x(1-x)w''+\left(c-(a+b+1)x\right)w'-abw=0,
\end{equation}
the solutions of this equation are the hypergemetric functions . We
observe that the hypergeometric differential equation admits a
polynomial solution if $a$ or $b$ are non-positive integer (for more
details see \cite{Abr} ).

It is well known that  orthogonal polynomials $f_0,\,f_1,\,\ldots
f_n,\,\ldots$ are the ones satisfying the differential equation (for
more details see for instance \cite{Abr})
\begin{equation}\label{RR01}
\tau_2(x)f''+\tau_1(x)f'+\tau_0f=0,
\end{equation} where
$\tau_j=\tau_j(x)=\displaystyle\sum_{n=1}^j\tau_{jn}x^j$ are a
polynomials of degree at most $j$, for $j=0,1,2.$
 The solution of \eqref{RR01}  is an  {\it orthogonal polynomial} if
one of the following sets of conditions hold:
\begin{itemize}
\item[(1)] The polynomial $p_2(x)$ is  quadratic with two distinct real roots, the
root of polynomial $r=r(x)$  lies strictly between the roots of
$p_2,$ and the leading terms of $p_2$ and $r$ has the same sign.
This case leads to the Jacobi-like polynomials which are solutions
of the differential equation
\begin{equation}\label{AS1}
(1-x^2)f''+(A-B-(A+B+2)x)f'+n(n+A+B+1)f=0,
\end{equation}where $A,B$ are real constants and $n$ is natural
number. Important special cases of Jacobi polynomials are Gegenbauer
polynomials ( with parameter $\gamma=A+1/2$), Legendre polynomials
($A=B=0$) and Chebyshev polynomials ($A=B=\pm 1/2$).

\item[(2)] The polynomial $p_2(x)$ is linear. The roots of $p_2$ and
$r$ are different, and the leading terms of $p_2$ and $r$ has the
same sign if the root of $r$ is less that the root of $p_2$, or
vice-versa. This case leads to the Laguerre-like polynomials which
are solutions of the differential equation
\begin{equation}\label{AS2}
xf''+(A+1-x)f'+nf=0,
\end{equation}where $A$ is a  real constants and $n$ is natural
number.
\item[(2)] $p_2$ is just a nonzero constant. The leading term of $r$
has the opposite sign of $p_2.$ This case leads to the Hermite-like
polynomials (see for instance \cite{Abr}) which are solutions of the
differential equation
\begin{equation}\label{AS3}
f''-xf'+nf=0,
\end{equation}where  $n$ is natural
number.
\end{itemize}
\smallskip

\begin{conjecture}\label{C1}
Differential  system \eqref{Rf4} admits  a Liouvillian first
integrals.
\end{conjecture}
\begin{proposition}\label{HypEquation}
Under the assumptions of Theorem \ref{Rf4} we obtain that the most
general quadratic vector field which admits as a solution the curve
\eqref{ORT} with $a_0$ the hypergeometric functions with $a$ or $b$
are non-positive integer, is
\begin{equation}\label{Rrff4}
\begin{array}{rl}
\dot{x}=&x(1-x),\vspace{0.2cm}\\
\dot{y}=&y^2+\left( 2\beta-a-b+1\right)xy+\left(c-1+2\gamma\right)y+\gamma(\gamma+c-1)\vspace{0.2cm}\\
&+\left(\beta^2+(2-a-b)\beta+(a-1)(b-1)
\right)x^2\vspace{0.2cm}\\
&+\left(c\beta-ab+(1-\gamma)(-2\beta-1+a+b)\right)x
\end{array}
\end{equation}
\end{proposition}

\begin{corollary}\label{Val1}
The most general quadratic vector field which admits as a solution
the curve  \eqref{ORT} with $a_0$ the Jacobi polynomial is
\begin{equation}\label{Nat}
\begin{array}{rl}
\dot{x}=&1-x^2,\\
\dot{y}=&y^2+(2\beta-A-B)xy+(2\gamma+A-B)y\\
&+\left((\beta^2+\beta-n(n+1)-(1+n+\beta)(A+B)\right)x^2\\
&+\left((\beta-\gamma)A-(\gamma+\beta)B+2\gamma\beta\right)x\\
&+(n+1-\gamma)B+(n+1+\gamma)A+n(n+1)+\gamma^2-\beta .
\end{array}
\end{equation}
The system \eqref{Rrff4}    coincide with system \eqref{Nat} under
the change
\[\begin{array}{rl}
x\longrightarrow &\dfrac{x+1}{2},\quad
y\longrightarrow\dfrac{y}{2},\quad
t\longrightarrow\dfrac{t}{2},\quad \beta+2\gamma\longrightarrow
\gamma,\vspace{0.2cm}\\
a=&-n,\quad b=1+n+A+B,\quad c=1+A.
\end{array}
\]
\end{corollary}
\begin{corollary}\label{Val2}
The most general quadratic vector field which admits as a solution
the Laguerre polynomials  is
\begin{equation}\label{Rf81}
\begin{array}{rl}
\dot{x}=&x,\vspace{0.2cm}\\
\dot{y}=&y^2+(2\beta-1)xy+(2\gamma+A)\,y+(\beta-1)\beta\,x^2\\
&+(\beta(A+2\gamma-1)+n+1-\gamma)x +A\gamma+\gamma^2.
\end{array}
\end{equation}
The system \eqref{Rrff4} coincide with system \eqref{Rf81} under the
change
\[\begin{array}{rl}
x\longrightarrow \dfrac{x}{b},\quad
\beta\longrightarrow\,\beta\,b,\vspace{0.2cm}\\
a=-n,\quad c=A+1,
\end{array}
\]
and tend to infinity the parameter $b.$
\end{corollary}

\begin{corollary}\label{Val3}
The most general quadratic vector field which admits as a solution
the curve  \eqref{ORT} with $a_0$ the  Hermite polynomial is
\begin{equation}\label{AA}
\begin{array}{rl}
\dot{x}=&1,\\
\dot{y}=&y^2+(2\beta-1)xy+2\gamma y\\
&+(\beta^2-2\beta)x^2+2\gamma(\beta-1)x+ \gamma^2-\beta+2(n+1)
\end{array}
\end{equation}
\end{corollary}
differential system \eqref{Rf4} and its  particular cases admits
algebraic curves of arbitrary high degree.
\begin{theorem}\label{Fil}
Differential systems \eqref{Rrff4},\,\eqref{Nat},\,\eqref{Rf81} and
\eqref{AA} are Liouvillian integrable.
\end{theorem}

\section{Proof of the main results}

\begin{proof}[Proof of Theorem \ref{A1}]

In order to obtain a polynomial solutions of the equation
\eqref{Rf2}  we consider the particular case when the quadratic
polynomial $q_2=q_2(x)$ i such that
\begin{equation}\label{Rtf4}
 q_2=(\kappa+r')p_2 -(\beta x+\gamma )\left((\beta
-q_{11})x+\gamma-q_{10}\right),
\end{equation}
where $\kappa$ is a constant.

\smallskip

 Under this condition equation \eqref{Rf2} takes the form
 \eqref{Rr01}.   By compare  equation \eqref{Rr01} and \eqref{RR01} we obtain that
\begin{equation}\label{RR011}
\tau_2(x)=p_2(x),\quad q_{11}=\tau_{11}-2p_{22}+2\beta,\quad
q_{10}=\tau_{10}+2\gamma-p_{21},\quad \kappa=\tau_0
\end{equation}
By solving  equation \eqref{Rtf4} with respect to $q_{22},\,q_{21}$
 and $q_{20}$ and in view of \eqref{RR011} we deduce that
\begin{equation}\label{RRRf4}
\begin{array}{rl}
q_{22}=&2p^2_{22}-\left(\tau_{11}-\tau_0+3\beta\right)p_{22}+\beta^2+\beta\tau_{11}\vspace{0.2cm}\\
q_{21}=&2p_{22}-\left(2\beta-\tau_0+\tau_{11}\right)p_{21}+
\gamma\left(\tau_{11}-2p_{22}+2\beta\right)+\beta\tau_{10}\vspace{0.2cm}\\
q_{20}=&p_{20}\left(2p_{22}-\tau_{11}+\tau_0-\beta\right)+\gamma\left(\gamma+\tau_{10}-p_{21}\right),
\end{array}
\end{equation}
By inserting \eqref{RRRf4} into \eqref{Nat1} we obtain \eqref{Rf4}.
 \end{proof}
\begin{proof}[Proof of Corollary \ref{HypEquation}]
The hypergeometric functions are solutions of the differential
equations \eqref{HypEq}. Thus by comparing with \eqref{Rr01} we
deduce that (we take $\alpha=1$)
\[
q_{11}=2\beta-a-b+1,\quad q_{10}=2\gamma-1+c,\quad \kappa=-ab
\]
consequently from \eqref{RRRf4} we give that
\begin{equation}\label{00}\begin{array}{rl}
q_{20}=&\gamma(\gamma+c-1),\\
q_{21}=&(1-\gamma)(a+b-2\beta-1)+c\beta-ab,\\
q_{22}=&\beta^2+(2-a-b)\beta+(b-1)(a-1).
\end{array}
\end{equation}
 Inserting into \eqref{Rf4} we finally deduce
\eqref{Rrff4}.
\end{proof}

\begin{proof}[Proof of Corollary \ref{Val1}]
{F}rom \eqref{AS1} follows that $\tau_{11}=-A-B-2,$ $\tau_{10}=A-B$
and $\tau_0=n(n+A+B+1).$ Consequently, in view of \eqref{RR011} and
\eqref{RRRf4} we obtain
\[\begin{array}{rl}
p_{22}=&-1,\quad p_{21}=0,\quad p_{20}=1,\vspace{0.2cm}\\
q_{11}=&2\beta-A-B,\quad q_{10}=2\gamma+A-B,\vspace{0.2cm}\\
q_{22}=&\beta^2+\beta-n(n+1)-(1+n+\beta)(A+B),\vspace{0.2cm}\\
q_{20}=&\gamma^2-\beta+n(n+1)+(1+n+\gamma)A+(n+1-\gamma)B.
\end{array}
\]
Inserting into \eqref{Rf4} we have the proof of the corollary.
\end{proof}
\begin{proof}[Proof of Corollary \ref{Val2}]
The Laguerre polynomials are solutions of the differential equations
\eqref{AS1}, thus from \eqref{Rr01} we have that
\[
q_{11}=-1+2\beta,\quad q_{10}=1+2\gamma,\quad \kappa=-n
\]
consequently in view of \eqref{00} we deduce that
\[
q_{20}=\gamma+\gamma^2,\quad q_{21}=2\gamma\beta+1-n-\gamma,\quad
q_{22}=\beta^2-\beta.
\]
Inserting into \eqref{Rf4} we finally deduce \eqref{Rf81}.
\end{proof}
\begin{proof}[Proof of Corollary \ref{Val3}]
The Hermite polynomials are solutions of the differential equations
\eqref{AS1}, thus from \eqref{Rr01} we have that
\[
q_{11}=2\beta-1,\quad q_{10}=2\gamma,\quad \kappa=-\lambda
\]
consequently in view of \eqref{00} we deduce that
\[
q_{20}=\lambda-\beta+1+\gamma^2,\quad q_{21}=\gamma(2\beta-1),\quad
q_{22}=\beta^2-\beta.
\]
Inserting into \eqref{Rf4} we finally deduce \eqref{Rf81}.
\end{proof}
\section{Proof of Theorem \ref{Fil}}

\begin{proof}[Proof of Theorem \ref{Fil}]
After some computations it is possible to show that  quadratic
differential system \eqref{Rrff4} admits the following first
integral
\[F=\dfrac{x^{1-c}g_1(x,y)}{g_2(x,y)},\]
where
\[
\begin{array}{rl}
g_1=&\left(y+(1-b+\beta)x+\gamma+c-1-a\right)F(1+a+c.1+b-c,2-c;x)\vspace{0.2cm}\\
&+(1+a-c)(1-x)F(1+a,b,c;x)+(1+a-c),\vspace{0.2cm}\\
g_2=&\left(y+(1-b+\beta)x+\gamma+c-1-a\right)F(a,b,c;x)\vspace{0.2cm}\\
&+a(1-x)F(2+a+c.1+b-c,2-c;x)+(1+a-c)
\end{array}
\]
where $F(a,b,c;x):= \,_2F_1(a,b,c;x)$ is the hypergeometric
function, where $2$ refers to number of parameters in numerator and
 1 refers to number of parameters in denominator.

 \smallskip

 In particular, differential system \eqref{Rrff4} for the values of the parameters
\begin{equation}\label{sst}
\beta=a+b-\dfrac{ab}{c}-1,\quad \gamma=1-c,
\end{equation}
takes the form
\[\begin{array}{rl}
\dot{x}=&x(1-x),\vspace{0.2cm}\\
\dot{y}=&y^2+(1-c)y+\left(a+b-1-\dfrac{2ab}{c}\right)xy+\dfrac{ab(b-c)(a-c)}{c^2}x^2.
\end{array}
\]
In \cite{CLL} was proved that this system which admits an algebraic
curves of arbitrary high degree is not rational integrable but is
Liouvillian integrable. Indeed, by considering that this system
admits four algebraic invariant curves
\[\begin{array}{rl}
g_1=&x=0,\quad g_2=x-1=0,\vspace{0.2cm}\\
g_3=&F_1\left(y-\dfrac{ab}{c}x\right)+x(1-x)F'_1=0\vspace{0.2cm}\\
g_4=&F_2\left(y-\left(\dfrac{ab}{c}+1-c\right)x-c+1\right)+x(1-x)F'_2=0,
\end{array}
\]
where $F_1=F(a,b,c,x),\,F_2=F(1+a-c,1+b-c,2-c,x)$ and $a$ is
negative integer. By considering that the cofactor are
\[K_1=x-1,\quad K_2=x,\quad K_3=y-\dfrac{(b-c)(a-c)}{c}x,\quad
K_4=y+\left(b+a-1-\dfrac{ba}{c}\right)x+1-c,\]respectively,  it is
easy to obtain the existence of the Liouvillian first integral
\[F=\dfrac{x^{c-1}g_3}{g_4}\] here $c$ is a
positive irrational number.

\smallskip

In view of the respectively corollary we have that the quadratic
systems \eqref{Nat} and \eqref{Rf81} are Liouvillian integrable.

\smallskip

The integrability of  system \eqref{AA} follows from the fact that
this system can be rewritten as follows
\[
\dfrac{d y}{dx}=y^2+(2\beta-1)xy+2\gamma y
+(\beta^2-2\beta)x^2+2\gamma(\beta-1)x+ \gamma^2-\beta+2(n+1)
\]
which is the Riccati  equations with one solution
\[g_1=\left(y+(\beta-1)x+\gamma\right)H(x)+H'(x)=0,\]
where $H$ is the Hermite polynomial \cite{CG}. In short the theorem
is proved.
\end{proof}

\subsection*{Acknowledgments}
This work was partly supported by the Spanish Ministry of Education
through projects DPI2007-66556-C03-03, TSI2007-65406-C03-01
"E-AEGIS" and Consolider CSD2007-00004 "ARES".


\end{document}